\documentclass[10pt,a4paper,english]{amsart}
\usepackage[T1]{fontenc}
\usepackage[utf8]{inputenc}
\usepackage[normalem]{ulem}
\usepackage{enumitem}

\usepackage{amssymb}
\usepackage{xcolor}
\definecolor{darkred}{RGB}{180,0,0}
\definecolor{darkblue}{RGB}{0,0,180}
\usepackage{hyperref}
\hypersetup{
  colorlinks = true,
  urlcolor   = darkblue,
  linkcolor  = darkblue,
  citecolor  = darkred,
  pdftitle={Super-expanders and warped cones by Damian Sawicki},
  pdfauthor={Damian Sawicki},
  pdfsubject={Super-expanders and warped cones},
}
\usepackage[initials]{amsrefs}
\usepackage[capitalise]{cleveref}
\newtheorem{thm}{Theorem}[section]
\newtheorem{theorem}[thm]{Theorem}
\newtheorem{lem}[thm]{Lemma}
\newtheorem{fact}[thm]{Fact}

\newtheorem{cor}[thm]{Corollary}

\theoremstyle{definition}
\newtheorem{defi}[thm]{Definition}

\def\R{{\mathbb R}}
\def\N{{\mathbb N}}
\def\Z{{\mathbb Z}}

\def\cO{{\mathcal O}}
\def\T{{\mathbb T}}
\def\eps{\varepsilon}
\newcommand*\dd{\mathop{}\!\mathrm{d}}

\DeclareMathOperator{\id}{id}
\DeclareMathOperator{\im}{im}
\DeclareMathOperator{\diam}{diam}
\DeclareMathOperator{\dom}{dom}

\newcommand{\acts}{\ensuremath\mathrel{\raisebox{.6pt}{$\curvearrowright$}}}
\newcommand*{\defeq}{\mathrel{\vcenter{\baselineskip0.5ex \lineskiplimit0pt
                     \hbox{\scriptsize.}\hbox{\scriptsize.}}}%
                     =}
\newcommand*{\eqdef}{=\mathrel{\vcenter{\baselineskip0.5ex \lineskiplimit0pt
                     \hbox{\scriptsize.}\hbox{\scriptsize.}}}}

\usepackage{tikz}
\newcommand*{\myhash}{%
  \begin{tikzpicture}
    \pgfmathsetlengthmacro\myWidth{.9*.8*width("=")}%
    \pgfmathsetlengthmacro\myHeight{.9*height("H")}%
    \pgfmathsetlengthmacro\mySepY{.37*\myHeight}%
    \pgfmathsetlengthmacro\mySideBearing{.1*\myWidth}%
    \def\myAngle{74}%
    \pgfmathsetlengthmacro\mySepX{\mySepY/sin(\myAngle)}%
    \pgfmathsetlengthmacro\mySlantX{\myHeight*cot(\myAngle)}%
    \pgfmathsetlengthmacro\mojaPoprawka{.37*\mySlantX}%
    \draw[line cap=round]
      (\mojaPoprawka, {(\myHeight + \mySepY)/2}) -- ++(\myWidth, 0);
    \draw[line cap=round]
      (0, {(\myHeight - \mySepY)/2}) -- ++(\myWidth, 0)

      ({(\myWidth - \mySepX - \mySlantX)/2+\mojaPoprawka/2}, 0)
      -- ({(\myWidth - \mySepX + \mySlantX)/2+\mojaPoprawka/2}, \myHeight)
      ({(\myWidth + \mySepX - \mySlantX)/2+\mojaPoprawka/2}, 0)
      -- ({(\myWidth + \mySepX + \mySlantX)/2+\mojaPoprawka/2}, \myHeight)
    ;%
    \useasboundingbox
      (-\mySideBearing, 0)
      (\myWidth + \mySideBearing, \myHeight)
    ;%
  \end{tikzpicture}%
}

\title{Super-expanders and warped cones}
\dedicatory{French title: Superexpanseurs et c\^ones tordus}
\author{Damian Sawicki}

\address{Institute of Mathematics\newline
Polish Academy of Sciences\newline
\'Sniadeckich 8\newline
00-656 Warszawa (Poland)\newline}

\curraddr{\newline
Max-Planck-Institut f\"{u}r Mathematik\newline
Vivatsgasse 7\newline
53111 Bonn (Germany)}

\urladdr{guests.mpim-bonn.mpg.de/dsawicki/}


\begin{document}
\begin{abstract}
For a Banach space $X$, we show that any family of graphs quasi-isometric to levels of a warped cone $\mathcal O_\Gamma Y$ is an expander with respect to $X$ if and only if the induced $\Gamma$-rep\-re\-sen\-ta\-tion on $L^2(Y;X)$ has a spectral gap. This provides examples of graphs that are an expander with respect to all Banach spaces of non-trivial type.

\smallskip \noindent \textsc{In French.} Pour un espace de Banach $X$, on montre que toute famille de graphes, quasi-isom\'etriques \`a des niveaux d'un c\^one tordu $\mathcal O_\Gamma Y$, est un expanseur relativement \`a $X$, si et seulement si la $\Gamma$-re\-pr\'e\-sen\-ta\-tion induite sur $L^2(Y;X)$ a un trou spectral. Ceci fournit des examples de graphes qui sont un expanseur relativement \'a tous les espaces de Banach de type non trivial.
\end{abstract}

\maketitle

\section{Results}

The aim of this paper is to prove expansion with respect to large classes of Banach spaces for new families of graphs. The graphs are 1-skeleta of Rips complexes of a family of compact metric spaces $(tY,d_\Gamma)_{t > 0}$, which originates in an action of a finitely generated group $\Gamma$ on a compact metric space $Y$. Such families were introduced by J.\ Roe under the name of warped cones \cite{Roe-cones}, and the idea of studying graphs quasi-isometric to them comes from Vigolo \cite{Vigolo}.

We prove that the spectral gap of the Koopman representation on the Bochner space $L^2(Y,\mu;X)$ is equivalent to the ex\-pan\-sion of these graphs with respect to the Banach space $X$. Since expansion is a quasi-isometry invariant (\cref{expansion-invariant}), the same can be said about any family of graphs quasi-isometric to the warped cone~$\cO_\Gamma Y$.

\begin{theorem}\label{main}
Let $(Y,d,\mu)$ be a geodesic Ahlfors regular compact metric measure space with an action by Lipschitz homeomorphisms of a finitely generated group $\Gamma$ and $X$ be any Banach space.

Then, the action has a spectral gap in $L^2(Y,\mu;X)$ if and only if the family $(tY,d_\Gamma)_{t >0} = \cO_\Gamma Y$ is quasi-isometric to an expander with respect to $X$.
\end{theorem}

In particular, this dynamical construction provides a large new class of expanders with respect to Banach spaces of non-trivial type. Such expanders (the strongest expanders presently known) have been previously obtained among quotients of groups --- see the seminal papers of V. Lafforgue \cites{Lafforgue1, Lafforgue2} and the generalisation of Liao \cite{Liao}.

\begin{cor}\label{cor} Let $\Gamma\acts (Y,d,\mu)$ be a measure-preserving ergodic action as in \cref{main}, and assume that $\Gamma$ has Lafforgue's Banach property (T).

Then, any family of bounded degree graphs quasi-isometric to $\cO_\Gamma Y$ is an expander with respect to all Banach spaces of non-trivial type.
\end{cor}

In order to address some recent developments \cites{Bou,BY,local} regarding local versions of the spectral gap condition, we verify \cref{main} not only for an action of a group generated by a finite set $S$ but also for a finite family $S$ of partially-defined bi-Lipschitz maps (\cref{main plus}).

The notion of expanders or \emph{non-linear spectral gaps} of \cite{Mendel-Naor} has been extensively studied also in the case when $X$ is a metric space (see, for instance, \cite{Mendel-Naor-Hadamard} and references therein). \cref{main} may also find applications in this context because it remains true under very mild assumptions about the space~$X$ (namely, density of $C(Y;X)$ in $L^2(Y,\mu;X)$), which are satisfied by large classes of metric spaces, notably CAT(0) (Hadamard) spaces considered in \cite{Mendel-Naor-Hadamard}.

\subsection*{Super-expanders}
\smallskip Roughly speaking, a sequence $(G_n)$ of graphs is an expander with respect to a Banach space $X$, if for every map $f\colon G_n\to X$ the norm of its global coboundary (i.e.\ the average of $\|f(v) - f(w)\|$ over $v,w\in G_n$) can be estimated by its local gradient (i.e.\ the average of $\|f(v) - f(w)\|$ over $\{v,w\}$ forming an edge). The former average involves a quadratic number of terms, while --- for sparse graphs --- the latter only a linear number, which is a significant computational advantage (cf.\ \cite{Mendel-Naor-Hadamard} and references therein).

The comparability of the global and local statistics implies in particular that a sequence of such maps $f_n\colon G_n\to X$ cannot preserve the metric structure of the graphs $G_n$, that is, it is never a bi-Lipschitz or coarse embedding (an observation originally due to Gromov). Yu proved the Novikov conjecture for groups admitting a coarse embedding into a Hilbert space \cite{A}, and Kasparov and Yu proved the coarse Novikov conjecture for bounded geometry metric spaces admitting a coarse embedding into a uniformly convex Banach space or, equivalently, into a super-reflexive Banach space \cite{Kasparov-Yu-super}. Hence, it was natural for Kasparov and Yu to ask whether there exist expanders with respect to all super-reflexive Banach spaces, or \emph{super-expanders} for short.

The positive answer was given by V. Lafforgue \cite{Lafforgue1}, who obtained super-ex\-panders as finite quotients of groups with his Banach property (T). Subsequently, he generalised these results from super-reflexive Banach spaces to an even larger class of Banach spaces, namely Banach spaces of non-trivial Rademacher type (or, equivalently, $K$-convex Banach spaces) \cite{Lafforgue2}. This is the largest class of Banach spaces for which expanders have been constructed.

Mendel and Naor \cite{Mendel-Naor} gave an independent construction of expanders with respect to super-reflexive Banach spaces, using the zig-zag product construction of \cite{zig-zag}.

The aim of this work is to provide another construction of super-ex\-panders, including expanders with respect to Banach spaces of non-trivial type.

\subsection*{Previous work} 
Warped cones were introduced by J.\ Roe \cite{Roe-cones} as a construction relating compact dynamical systems with large scale geometry. Dru\k{t}u and Nowak \cite{Drutu-Nowak} noticed that a particularly interesting case is when the system has a spectral gap, conjecturing that it may yield new counterexamples to the coarse Baum--Connes conjecture.

Following this direction, Nowak and the author \cite{Nowak-Sawicki} proved that for an action $\Gamma \acts Y$ with a spectral gap in $L^2(Y,\mu;X)$, the warped cone $\cO_\Gamma Y$ does not embed coarsely into $X$. If in addition the measure is Ahlfors regular and the action is by Lipschitz homeomorphisms, they showed that the $X$-distortion of levels $tY \in \cO_\Gamma Y$ is logarithmic.

Subsequently, assuming a condition equivalent to a spectral gap in $L^2(Y,\mu;\ell^2)$, Vigolo \cite{Vigolo} showed under similar hypotheses that the levels $tY$ are quasi-iso\-met\-ric to an expander (this implies the above result for $X$ being a Hilbert space) and that it is a sufficient and necessary condition.

\cref{main} combines advantages of the main results of \cite{Nowak-Sawicki} and \cite{Vigolo}: on the one hand, it applies to all Banach spaces, and on the other, it gives a characterisation of expansion rather than a criterion for distortion or non-embeddability.

The conjecture of Dru\k{t}u and Nowak is established in the forthcoming work \cite{Sawicki-cBCc} without requiring $Y$ to be geodesic --- a restriction important in \cite{Vigolo} and the present work --- yielding in particular counterexamples to the coarse Baum--Connes conjecture not coarsely equivalent to any family of graphs.

\subsection*{Further developments} Let us briefly report on some advancements in the area since the first version of this note was made public. The progress regards rigidity properties of warped cones, showing in particular that our construction is a very rich source of super-expanders and that expanders with respect to Banach spaces of non-trivial type provided by the construction are indeed provably distinct from (i.e.\ not quasi-isometric to) those previously known.

Jianchao Wu and the author \cite{SW} proved that for free actions the warped cone $\cO_\Gamma Y$ admits an \emph{asymptotically faithful} covering by products $\Gamma\times Y$, where the metric on $Y$ is appropriately scaled. By an argument of Khukhro and Valette \cite{KV}, for an $n$-manifold $Y$ it implies that a quasi-isometry $\cO_\Gamma Y\simeq (\Lambda/\Lambda_i)$ with a Margulis expander $(\Lambda/\Lambda_i)$ induces a quasi-isometry of groups
\begin{equation}\label{induced qi}
\Gamma\times \Z^n \simeq \Gamma\times\R^n \simeq \Lambda,
\end{equation}
see \cites{Sawicki-piecewise, dLV} for details. As noticed in \cite{dLV}, it follows from quasi-isometric-rigid\-ity literature that our super-expanders cannot be quasi-isometric to many super-expanders of Lafforgue and Liao.

David Fisher, Thang Nguyen, and Wouter van Limbeek \cite{FNvL} constructed a continuum of actions $\Gamma\acts Y$ yielding pairwise non-quasi-isometric warped cones satisfying the assumptions of \cref{main} for all Banach spaces of non-trivial type. This is a consequence of an impressive \emph{dynamical} quasi-isometric-rigidity result for warped cones that they obtain. An important tool employed are coarse fundamental groups, which were also independently studied in a very interesting work \cite{fundamental Vigolo} by Federico Vigolo.

It follows from computations in \cites{FNvL, fundamental Vigolo} that, for $Y$ with a finite fundamental group, the coarse fundamental group of $\cO_\Gamma Y$ is virtually isomorphic to $\Gamma$. If $\cO_\Gamma Y$ is quasi-isometric to a Margulis expander $(\Lambda/\Lambda_i)$, then combining \eqref{induced qi} with the induced virtual isomorphism $\Gamma \simeq \Lambda$ of coarse fundamental groups, one gets a quasi-isometry $\Gamma\simeq \Gamma\times \Z^n$, impossible for many $\Gamma$. This way Vigolo \cite{fundamental Vigolo} proves that there exist warped cones satisfying the assumptions of \cref{cor} that are not quasi-isometric to \emph{any} Margulis expander.

\subsection*{Methods and outline of the paper} Our treatment is functional-analytic in flavour and, like \cite{Nowak-Sawicki}, utilises Poincar\'e inequalities rather than isoperimetric inequalities, which are best-suited for classical (Hilbert-space) expanders. Nonetheless, it picks up some ideas from the measure-theoretic approach of \cite{Vigolo}, notably the key idea of studying families of graphs quasi-isometric to a warped cone, which allows the language of graph theory.

In \cref{sec:expanders} we introduce the definitions of an expander and an expander with respect to a Banach space and discuss methods of constructing them. \Cref{sec:gaps} defines a spectral gap via a Poincar\'e-type inequality, relates it to other definitions in the literature, and briefly reminds some examples.

\Cref{sec:proofs} recalls the definition of a warped cone and the quasi-isometric invariance of expansion (\cref{expansion-invariant}) and constructs certain graphs quasi-isometric to the warped cone. We give the core of the argument in \cref{subsec:Poincare}. The basic idea is to consider the above graphs --- whose vertex sets are finer and finer partitions of $Y$ --- and to reinterpret functions defined on them as functions defined on $Y$ or to approximate $f\in L^2(Y, \mu;X)$ by simple functions $f_n$ subordinate to these partitions. Then, we compare the norms of the coboundary of $f$ (i.e.\ $F(u,v)\defeq f(u)-f(v)$ for $(u,v)\in Y\times Y$) and the coboundary of $f_n$ with the norms of a certain `gradient' of $f$ with respect to the action and the graph-theoretic gradient of $f_n$. Finally, \cref{subsec:remarks} contains some remarks.

\section{Expanders}\label{sec:expanders}

Expanders are graphs that are simultaneously sparse and highly connected.

\begin{defi}\label{Cheeger} Let $G=(V,E)$ be a finite graph. Its \emph{Cheeger constant} is defined as
\[h(G) = \min \left\{ \frac{\myhash  E(A,V\setminus A)}{\min(\myhash  A, \myhash  (V\setminus A))} : A\subseteq V\text{ with }A\neq \emptyset,V \right\},\]
where $E(A, A')$ denotes the set of edges joining elements of subsets $A$ and $A'$ of $V$, and $\myhash  X$ is the cardinality of $X$.

An infinite family of finite graphs $(G_i)_{i\in I}$ is \emph{an expander} if their cardinalities are unbounded, vertex degrees uniformly bounded above, and Cheeger constants uniformly bounded away from $0$.
\end{defi}

Expanders are not all the same. One approach to compare `strength' of expanders is to consider constants measuring expansion (like the Cheeger constant~$h$ from \cref{Cheeger}), see \cites{Ramanujan, Margulis-Ramanujan}. Another approach is to rephrase the definition so that it is parametrised by a Banach space, and then measure the `strength' of an expander by the `size' of the class of Banach spaces with respect to which it is an expander.

\begin{defi}\label{super defi} Let $X$ be a Banach space and $(G_i)_{i\in I}$ be a family of finite graphs with uniformly bounded vertex degrees and unbounded cardinalities. We say that $(G_i)$ is \emph{an expander with respect to $X$} if there exists a constant $\eta>0$ such that
\begin{equation}\label{expansion inequality}
\sum_{x\sim y}\|f(x) - f(y)\|_X^2 \geq \frac{\eta}{\myhash G_i}\sum_{x,y\in G_i} \|f(x) - f(y)\|_X^2
\end{equation}
for any $i\in I$ and any function $f\colon G_i \to X$, where we sum over pairs of vertices of $G_i$, and $x\sim y$ means that there is an edge joining $x$ and $y$.
\end{defi}

For more on Banach-space expanders, consult \cites{Mendel-Naor,Lafforgue1,Lafforgue2,Cheng,Mimura}. By replacing norm distances $\|f(x)-f(y) \|_X$ by distances $d_X(f(x), f(y))$ with respect to some metric~$d_X$, we can generalise \cref{super defi} to the case where $(X,d_X)$ is a metric space. Note that, as long as $X$ contains at least two points, an expander with respect to $X$ is an expander in the sense of \cref{Cheeger} and --- in the other direction --- classical expansion implies expansion with respect to any Hilbert space. In this generality, Mendel and Naor proved that indeed the `strength' of expanders, measured by spaces $X$, may differ \cite{Mendel-Naor-Hadamard}.

Because of all the applications of expanders \cites{hash,HoLiWi,HLS,exp in pure and applied}, the problem of constructing them has attracted a great deal of attention \cites{zig-zag, HoLiWi}. The first explicit construction due to Margulis and many subsequent constructions used Cayley (or Schreier) graphs of groups, relying on group-theoretic arguments, like (relative) property (T) of Kazhdan \cite{Margulis} or additive combinatorics \cites{BG0}. In particular, the super-expanders of Lafforgue are constructed with the approach of \cite{Margulis}. There is also a purely combinatorial construction via `zig-zag graph products' of Reingold, Vadhan, and Wigderson \cite{zig-zag}, which was employed in the construction of super-expanders by Mendel and Naor.

Our construction follows another approach, which dates back to 1980s in some special cases \cite{Gabber-Galil}. More recently, it was applied by Bourgain and Yehudayoff, who defined a continuous expander as a family of smooth maps $\Psi$ from a subinterval of $Y=[0,1]$ to $Y$ such that for every measurable $A\subseteq Y$ with $\mu(A) \leq 1/2$, we have:
\begin{equation}\label{continuous expansion}
\mu\Big(\bigcup\nolimits_{\psi\in \Psi} \psi(A)\Big) \geq (1+c)\, \mu(A)
\end{equation}
for some positive $c$ not depending on $A$ (where $\mu(A)$ denotes the Lebesgue measure of~$A$) \cites{Bou, BY}. By `discretising' such a continuous expander, they obtained the first \emph{monotone} expander (see also \cite{local}).

For much more general metric measure spaces $Y$ equipped with a group action, Vigolo proved that the family of graphs resulting from a similar discretisation is an expander if \emph{and only if} $Y$ with a finite family $\Psi$ coming from the action is a continuous expander \cite{Vigolo} (the `only if' implication is the more difficult one).

The discretising procedure involves non-canonical choices. By using warped cones --- which corresponds to adding more edges to the discretisations, called \emph{lavish} approximating graphs by Vigolo \cite{Vigolo} --- one obtains graphs that are unique up to quasi-isometry because they are essentially $1$-skeleta of Rips complexes for a fixed family of metric spaces. As already mentioned in the introduction, this novel treatment allows applying the rapidly developing theory of warped cones to understand the geometry of such expanders.

The above results \cites{Gabber-Galil, Bou, BY, local, Vigolo} concerned only classical expanders (as in \cref{Cheeger}). In terms of the involved tools, these articles relate the isoperimetric inequality \eqref{continuous expansion} to the isoperimetric characterisation of expanders via Cheeger constants (\cref{Cheeger}). Relating Poincar\'e inequalities --- apart from being a new technique valid also in the classical setting --- allows obtaining expanders with respect to Banach spaces.

\section{Spectral gaps}\label{sec:gaps}

Our working definition of spectral gap, which will remove unnecessary technical difficulties from the proof of the main result, will be the following (in fact its slight generalisation, \cref{spectral gap ultimate plus}). The definition remains sound also for a metric space $X$, after replacing norm distances $\|f(u)-f(v)\|_X$ by metric distances $d_X(f(u),f(v))$, but, for the sake of exposition, we restrict ourselves to the case of Banach spaces.
\begin{defi}\label{spectral gap ultimate}
Let $X$ be a Banach space, $\Gamma\acts (Y,\mu)$ be a probability-measure-pre\-serving action, and $L^2(Y,\mu;X)$ be the space of Bochner-meas\-ur\-able square-integrable\footnote{For a metric space $(X,d_X)$, square-integrability of $f\colon Y\to X$ with $Y$ compact means that $f$ is at finite $L^2$-distance from any (equivalently: every) continuous function, where the distance is given by $d_{L^2}(f,g)^2 = \int_Y d_X(f(y),g(y))^2\dd y$.} $X$-valued functions. The action $\Gamma\acts Y$ \emph{has a spectral gap in $L^2(Y,\mu; X)$}, if there exists a constant $\kappa>0$ such that
\begin{equation}\label{inequality ultimate}
\sum_{s\in S} \int_Y \|f(u)-f(s^{-1}u)\|_X^2\dd u \geq \kappa \iint_{Y^2} \|f(u)-f(v)\|_X^2 \dd u\dd v,
\end{equation}
for any function $f\in L^2(Y,\mu;X)$.
\end{defi}

Note that a spectral gap for any non-zero Banach space $X$ implies a spectral gap with $X=\ell^2$, which is equivalent to `expansion in measure' of \cite{Vigolo}.

Let $L^2_0(Y,\mu;X) \subseteq L^2(Y,\mu;X)$ denote the subspace of functions $f$ such that $\int f \dd\mu = 0$. Observe that it suffices to check inequality \eqref{inequality ultimate} for $f\in L^2_0(Y,\mu; X)$ because adding a constant to $f$ does not affect the inequality. Another, more standard, characterisation of the spectral gap condition says that there are no almost invariant vectors in $L^2_0(Y,\mu; X)$ for the Koopman representation $\pi$; that is, there exists a constant $\kappa>0$ such that for every $f\in L^2_0(Y,\mu;X)$ we have
\begin{equation}\label{inequality gap for Banach}
\max_{s\in S} \| f - \pi_s f \| \geq \kappa \|f\|.
\end{equation}

Indeed, note that the square of the left-hand side of \eqref{inequality gap for Banach} is proportional to the left-hand side of \eqref{inequality ultimate} (with the proportionality constants depending on the cardinality of~$S$). Thus, in order to confirm that the above condition is equivalent to the one from \cref{spectral gap ultimate}, it suffices to check that for $f\in L^2_0(Y,\mu; X)$ the expression $\|f\|^2$ is comparable to the integral $\iint_{Y^2} \|f(u) - f(v)\|_X^2 \dd u\dd v.$

\begin{lem} Let $(Y,\mu)$ be a probability space, $X$ be any Banach space, $f\in L^2_0(Y,\mu;X)$, and $F\in L^2_0(Y\times Y,\mu\times \mu;X)$ be given by the formula $F(u,v) = f(u) - f(v)$. Then:
\[\tfrac{1}{2} \| f \| \leq \| F \| \leq 2 \| f \|.\]
\end{lem}
\begin{proof}
The latter inequality follows from the triangle inequality: for $F_1,F_2\in L^2(Y\times Y,\mu\times \mu;X)$ given by $F_1(u,v) = f(u)$, $F_2(u,v)= f(v)$, we have $\|F_1\|=\|F_2\| = \|f\|$ and $F=F_1-F_2$.
For the former inequality observe that the projection onto constants $P\colon L^2(Y,\mu;X)\mapsto X \subseteq L^2(Y,\mu;X)$ given by integration
\[\big(P(f)\big)(u) = \int_Y f(y) \dd y \]
is a contraction from $L^1$ and hence also a contraction from $L^2$, as our measure is probabilistic. Thus we know that the complement projection $P_0 = \operatorname{Id} - P$ onto $L^2_0(Y,\mu;X)$ has norm at most $2$. Consequently:
\begin{align*}
4\|F\|^2
&= 4\iint \|f(u) - f(v)\|_X^2 \dd u \dd v 
= 4 \int \|f-f(v)\|^2\dd v \\
&\geq \int \|P_0(f - f(v))\|^2\dd v
= \int \|f\|^2\dd v = \|f\|^2.\qedhere
\end{align*}
\end{proof}

The following gives a yet another characterisation.

\begin{fact}[Dru\k{t}u--Nowak \cite{Drutu-Nowak}]\label{fact-DN} Assume that $X$ is a uniformly convex Banach space and the finite generating set $S$ for $\Gamma$ contains the identity element. Then the action $\Gamma\acts Y$ has a spectral gap in $L^2(Y,\mu;X)$ if and only if we have
\begin{equation}\label{inequality spectral}
\bigg\| \frac{1}{\myhash  S} \sum_{s\in S} \pi_s \bigg \| < 1,
\end{equation}
where the above operator norm is considered in the space of bounded operators on $L^2_0(Y,\mu;X)$.
\end{fact}

Note that the above Markov operator $M=\frac{1}{{\scriptsize \myhash}  S} \sum_{s\in S} \pi_s$ --- when considered as an operator on the whole of $L^2(Y,\mu;X)$ --- preserves constant functions. Hence, if it satisfies \eqref{inequality spectral}, its spectrum intersects a certain neighbourhood of $1$ only at $1$, justifying the name `spectral gap' in \cref{spectral gap ultimate}. In fact, the `if' implication of \cref{fact-DN} does not require $X$ to be uniformly convex, so for arbitrary Banach spaces our definition of the spectral gap may be more general than condition \eqref{inequality spectral}.

\cref{spectral gap ultimate} can be slightly generalised. The following definition is intended as an intrinsic version of the \emph{local} spectral gap with respect to $Y\subseteq Z$ of an action $\Gamma\acts Z$ --- a condition introduced and studied in \cite{local}.

\begin{defi}\label{spectral gap ultimate plus}
Let $X$ be a Banach space, $(Y,\mu)$ a probability space, and $S$ a finite set such that every $s\in S$ is a measurable bijection $s\colon \dom s\to \im s$ between measurable subsets of $Y$, and assume that $S$ is closed under inverses.

The set $S$ \emph{has a spectral gap in $L^2(Y,\mu; X)$} if there exists a constant $\kappa>0$ such that
\begin{equation}\label{inequality ultimate plus}
\sum_{s\in S} \int_{\im s} \|f(u)-f(s^{-1}(u))\|_X^2\dd u \geq \kappa \iint_{Y^2} \|f(u)-f(v)\|_X^2 \dd u\dd v,
\end{equation}
for any function $f\in L^2(Y,\mu;X)$.
\end{defi}

The main result of \cite{local} is that actions of dense subgroups generated by `algebraic elements' on \emph{non}-compact simple Lie groups have a local spectral gap. Restrictions of such actions give examples of sets $S$ with a spectral gap in the sense of \cref{spectral gap ultimate plus}.

For $s\in S$, following the convention for group actions, we will later denote the image of $y\in \dom s$ under $s$ by $sy$. For $U\subseteq Y$ we will shortly write $s U$ to denote the image of $U\cap \dom s$ under $s$.

\subsection{Examples} There are many sources of actions with a spectral gap, to which \cref{main} can be applied. Since \cite{Nowak-Sawicki}*{Section 4} and \cite{Vigolo}*{Section 8} already provide a detailed discussion, our treatment here is concise.
\begin{itemize}[itemsep=2pt,leftmargin=8mm,label={--}]
\item Any ergodic action $\Gamma\acts(Y,\mu)$ of a property (T) group has a spectral gap.
\item Some elementary actions are known to have a spectral gap, e.g.\ $\mathrm{SL}_2(\Z)\acts\T^2$.
\item For dense subgroups generated by `algebraic elements' in compact simple Lie groups, the action by left translations on the ambient group has a spectral gap by celebrated results of Bourgain--Gamburd and Benoist--de Saxc\'e \cites{BG, BG2, BdS}.
\item The above ordinary spectal gap (that is, in $L^2(Y,\mu;\R)$) for $\Gamma\acts Y$  is equivalent to a spectral gap in $L^p(Y,\mu;L^p)$ for any $1<p<\infty$ (see, for example, \cite{Nowak-Sawicki}).
\item Any ergodic action $\Gamma\acts(Y,\mu)$  of a group with V.\ Lafforgue's Banach property (T), robust property (T), or property (T$^\text{proj}$) with respect to a class of Banach spaces $\mathcal{E}\owns L_2(Y,\mu;X)$ (this is typically equivalent to $\mathcal E\owns X$) has a spectral gap in $L^2(Y,\mu;X)$ \cites{dlS,dlSdL,Lafforgue1,Lafforgue2,Liao,Oppenheim1,Oppenheim2}.
\end{itemize}

\section{Proof of the main results}\label{sec:proofs}

The following is an adaptation of the definition of J. Roe \cite{Roe-cones}.

\begin{defi}\label{warped} Let $(Y,d)$ be a compact geodesic metric space, and let $S$ be a closed under inverses finite set of Lipschitz homeomorphisms between closed subsets of $Y$.

\emph{The warped cone $\cO_S Y$} is a collection of metric spaces $(\{t\}\times Y)_{t > 0}$ (denoted $(tY)_{t >0}$ for brevity), where each set $tY$ is equipped with the largest metric $d_S$ such that
\[d_S(tx,ty) \leq t d(x,y) \quad\text{ and }\quad d_S(ty,tsy)) \leq 1,\]
where in the latter condition we consider all $s\in S$ and $y\in \dom s$.
\end{defi}

Since replacing the metric on $Y$ by a Lipschitz equivalent metric (see \cref{defi:QI}) yields Lipschitz equivalent warped cones, the assumption that $Y$ is geodesic can be satisfied in most cases of interest. We will assume that $Y$ is not a singleton.

Typically, in \cref{warped} one assumes that $S$ consists of homeomorphisms $Y\to \nolinebreak Y$. Then, if $S'$ is another such set generating the same group~$\Gamma$ of homeomorphisms, the respective warped cones are Lipschitz equivalent \cite{Roe-cones}. It justifies the standard convention of denoting the warped cone by $\cO_\Gamma Y$ and its metrics by $d_\Gamma$, which we conformed to in \cref{main} and \cref{cor}. With notation explained, it is clear that \cref{main} is a special case of the following result.

\begin{theorem}\label{main plus}
Let $(Y,d)$ and $S$ be as in \cref{warped}, and assume that $(Y,d,\mu)$ is an Ahlfors regular metric measure space and that $\mu(\partial \dom s)=0$ for all $s\in S$. Let $X$ be any Banach space.

Then, $S$ has a spectral gap in $L^2(Y,\mu;X)$ if and only if the family $(tY,d_S)_{t >0} = \cO_S Y$ is quasi-isometric to an expander with respect to $X$.
\end{theorem}

The assumption $\mu(\partial \dom s)=0$ on the measure of boundaries will only be relevant for the `if' implication.

\subsection{Quasi-isometries} Let us fix some terminology and notation.

\begin{defi}\label{defi:QI} A map $f\colon (Y,d_Y)\to (Z,d_Z)$ between metric spaces is \emph{a quasi-isometry} if there exist constants $A\geq 0$ and $C\geq 1$ such that
\begin{equation}\label{qi ineq}
C^{-1} d_Y(x,y) - A \leq d_Z(f(x), f(y)) \leq C d_Y(x,y) + A
\end{equation}
for all $x,y\in Y$ and such that the $B$-neighbourhood of $\im f$ is the whole of $Z$ for some $B>0$.

A bijective map $f\colon (Y,d_Y)\to (Z,d_Z)$ is \emph{a Lipschitz equivalence} (or simply \emph{bi-Lipschitz}) if one can take $A=0$ in \eqref{qi ineq}.
\end{defi}

\cref{defi:QI} applies respectively to families of maps $(f_t\colon Y_t\to Z_t)_{t\in T}$, where one requires the constants to be universal. Our index set $T$ will typically be $(0,\infty)$.

The following result is well-known among experts, at least for Hilbert-space expanders. We include the proof for the reader's convenience.

\begin{lem}\label{expansion-invariant} Expansion with respect to $X$ is a quasi-isometry invariant of families of graphs of uniformly bounded degree.
\end{lem}
\begin{proof}
Let $G=(G_t)$ and $H=(H_t)$ be two families of finite graphs with degrees bounded by $D<\infty$ and $i=(i_t\colon G_t\to H_t)$ be a quasi-isometry. That is, there exist positive constants $A$, $B$, $C$ such that
\[C^{-1}d(g,g')-A \leq  d(i(g),i(g'))\leq Cd(g,g') + A\]
and for every $h\in H_t$ there exists $g_h\in G_t$ with $d(h, i(g_h))\leq B$. Let us denote by $K_\text{f}$ the maximal size of a fibre of $i$, which is bounded by the maximal size of a (closed) $CA$-ball in $G$, by $K_B$ the maximal size of a $B$-ball in $H$, and by $K_{C+A}$ the maximal size of a $(C+A)$-ball in $H$. All three can be bounded in terms of $D$.

We will show that if $G$ is an expander with respect to $X$, then $H$ is too.

Let $f\colon H_t\to X$. We have:
\begin{align}
\sum_{h,k\in H_t} &\|f(h) - f(k)\|^2\notag\\
&\leq 3 \sum_{h,k\in H_t} \Big( \|f(h)-f(i(g_h))\|^2 + \|f(i(g_h))-f(i(g_k))\|^2 \notag\\
&\phantom{\leq}\; + \|f(i(g_k)) - f(k) \|^2 \Big) \notag\\ 
&\leq 3\sum_{h,k\in H_t} \|f(i(g_h))-f(i(g_k))\|^2 + 6\myhash  H_t \sum_{h\in H_t} \|f(h)-f(i(g_h))\|^2. \label{nr}
\end{align}
If one denotes a geodesic path from $h$ to $i(g_h)$ by $(h_l)_{l=0}^n$ (where $n\leq B$ depends on~$h$), then the second sum $\sum_{h\in H_t} \|f(h)-f(i(g_h))\|^2$ in \eqref{nr} is bounded by:
\begin{align*}
\sum_{h\in H_t} n \sum_{l=1}^n \| f(h_{l-1}) - f(h_l) \|^2 
\leq B K_B \sum_{h\sim k\in H_t} \|f(h)-f(k)\|^2.
\end{align*}
On the other hand, if we denote the composition $f\circ i$ by $f'$, then, by expansion of graphs $G_t$, the first sum in \eqref{nr} can be bounded above as follows:
\begin{align*}
\frac{\eta}{K_B^2}\sum_{h,k\in H_t} \|f(i(g_h))-f(i(g_k))\|^2
&\leq \eta \sum_{g,j\in G_t} \|f'(g)-f'(j)\|^2 \\
&\leq \myhash G_t \sum_{g\sim j} \|f'(g)-f'(j)\|^2,
\end{align*}
and --- after denoting a geodesic path from $i(g)$ to $i(j)$ by $(p^{g,j}_l)_{l=0}^m$ with $m\leq C+A$ --- we get:
\begin{align*}
\sum_{g\sim j} \|f'(g)-f'(j)\|^2
&\leq \sum_{g\sim j} m\sum_{l=1}^m \| f(p^{g,j}_{l-1}) - f(p^{g,j}_l) \|^2 \\
&\leq (C+A)K_\text{f}^2 K_{C+A}^2 \sum_{h\sim k} \|f(h)-f(k)\|^2.
\end{align*}
Summarising, we have just proved the lemma:
\begin{align*}
&\sum_{h,k\in H_t} \|f(h) - f(k)\|^2 \\
&\phantom{+}\leq \left(\frac{3(C+A)K_\text{f}^2 K_B^2 K_{C+A}^2 \myhash  G_t}{\eta} + 6\myhash  H_t B K_B\right) \sum_{h\sim k} \|f(h)-f(k)\|^2 \\
&\phantom{+}\leq \left(\frac{3(C+A)K_\text{f}^3 K_B^2 K_{C+A}^2}{\eta} + 6 B K_B\right) \myhash  H_t \sum_{h\sim k} \|f(h)-f(k)\|^2.\qedhere
\end{align*}
\end{proof}

\subsection{Construction of graphs}\label{construction}

Let $B(y,r)$ denote the open ball about $y$ of radius~$r$.

\begin{defi}\label{Ahlfors} A metric measure space $(Y,d,\mu)$ is \emph{Ahlfors regular} if there exist constants $c,m>0$ and $C\geq 1$ such that for any $r\leq 2\diam Y$ and $y\in Y$ we have
\[cr^m \leq \mu\big(B(y,r)\big) \leq Cc r^m.\]
\end{defi}

Let $(Y,d,\mu)$ and $S$ be as in \cref{main plus}.
Given $t>0$, pick a maximal $1/t$-separated set $Z = Z(t)$ in $Y$ and a respective measurable partition $\{U_z\}_{z\in Z}$ such that $B(z,1/(2t))\subseteq U_z\subseteq B(z,1/t)$. In particular, $\diam U_z \leq 2/t$ and
\[\frac{1}{K \myhash Z} \leq \mu(U_z)\leq \frac{K}{\myhash Z},\]
where $K=2^mC$ for $C$ and $m$ coming from the definition of Ahlfors regularity. Construct a graph $G=G(t)$ with the vertex set $Z$ and edges of two types, that is, there is an edge $y\sim z$ if $y\sim_1 z$ or $y\sim_2 z$, where the latter are defined by:
\begin{align*}
y\sim_1 z \iff &d(y,z)\leq 3/t,\\
y\sim_s z \iff &s U_y \cap U_z \neq \emptyset, \quad\text{where }s\in S,\\
y\sim_2 z \iff &\exists{s\in S}: y \sim_s z.
\end{align*}

The following lemma is a version of the fact that a Rips complex is quasi-isometric to the ambient space. Note that we could as well define $G$ as the 1-skeleton of the Rips complex for $Z$ with respect to the metric $d_S$, but the above definition, differentiating between edges of type (1) and (2), is more convenient in our proof. The `moreover' part of the lemma is another version of \cite{Roe-cones}*{Proposition 1.10} and \cite{Vigolo}*{Proposition~4.1}.
\begin{lem}\label{Rips-QI} Under the hypotheses of \cref{main plus}
the inclusions $Z(t) \subseteq Y$ induce bi-Lipschitz embeddings $G(t) \subseteq (tY,d_S)$ with the Lipschitz constants uniform in~$t$. In particular, the families $(G(t))_{t > 0}$ and $(tY)_{t > 0} = \cO_S Y$ are quasi-isometric. Moreover, the graphs $G(t)$ have vertex degrees bounded by some constant $D<\infty$ independent of~$t$.
\end{lem}
\begin{proof}
Fix $t>0$ and let $y\neq z$ be vertices of $G(t)$. Let $n\in \N$ be the minimal number such that $1\leq d_S(ty,tz) \leq n$ and observe that there exists a sequence of points $y=x_0,\ldots,x_{2n}=z\in Y$ with $td(x_{2i},x_{2i+1})\leq 1$ and $s x_{2i+1}= x_{2i+2}$ for some $s\in S\cup\{\id_Y\}$ (cf.\ \cite{Roe-cones}*{Proposition 1.6}). Denote by $z_i$ the unique element of $Z$ with $x_i\in U_{z_i}$ and observe that for every $i=1,\ldots,2n$ the pair $\{z_{i-1},z_i\}$ forms an edge. Hence, the distance $d_G(y,z)$ is bounded by $2n \leq 4 d_S(ty,tz).$
For the opposite bound it suffices to look at edges: if $y\sim_1 z$, then clearly $d_S(ty,tz) \leq td(y,z) \leq 3$; and if $y\sim_s z$, that is, there exists $x\in sU_y\cap U_z$, then also
\[d_S(ty,tz)\leq d_S(ty, s^{-1}tx) + d_S(s^{-1}tx, tx) + d_S(tx, tz) \leq 3,\]
which proves that the map $G(t)\to tY$ is a bi-Lipschitz embedding with constant $4\cdot 3 = 12$.

For the `moreover' part, observe that the number of edges of type (1) emanating from $y\in Z$ is bounded by the quotient $\frac{\mu(B(y,4/t))}{\inf_{z} \mu (B(z,1/(2t)))} \leq 8^m C$ (for constants $C,m$ from \cref{Ahlfors}) because the disjoint union of $U_{z}$ (each of measure at least $\mu\big(B(z,1/(2t))\big)$) over $d(y,z)\leq 3/t$  is contained in the ball $B(y,4/t)$. Similarly, for $s\in S$ and $y\in Z$ the number of edges $y\sim_s z$ is bounded by
\[\frac{\mu\big(B(sU_y,2/t)\big)}{\inf_{z} \mu(U_{z})} \leq \frac{\sup_z \mu\big(B(z,(2L+2)/t)\big)}{\inf_z \mu \big(B(z,1/(2t))\big)} \leq (4L+4)^m C,\]
for $L$ being the Lipschitz constant for the action of generators.
\end{proof}

\subsection{Poincar\'e inequalities}\label{subsec:Poincare}

After this preparation we are ready to prove the implications relating \eqref{expansion inequality} with \eqref{inequality ultimate plus}, that is, the implications relating Poincar\'e inequalities for maps $G(t)\to X$ with those for maps $Y\to X$. This, together with the results of the previous section, completes the proof of \cref{main plus}.

\begin{proof}[Expansion $\implies$ spectral gap] Assume that $\cO_S Y$ is quasi-isometric to an expander with respect to $X$. Since by \cref{Rips-QI} the family $(tY)_{t>0} = \cO_S Y$ is quasi-isometric to graphs $(G(t))$ having bounded degrees, and by \cref{expansion-invariant} being an expander is a quasi-isometry invariant among bounded degree graphs, we know that the family $(G(t))$ is an expander.

It suffices to check the spectral gap condition on the dense subset $C(Y;X)\allowbreak\subseteq L^2(Y,\mu;X)$ of continuous functions. Pick $f\in C(Y;X)$ and assume $0<F^2 \defeq \iint_{Y^2} \|f(u)-f(v)\|^2 \dd u\dd v$ (if $F=0$, then $f$ is constant and the inequality is trivially satisfied). For $t>0$ and the respective $Z$ and $\{U_z\}_{z\in Z}$ used to define $G(t)$ in \cref{construction}, we define $g\in L^2(Y;X)$ to be constant on every $U_z$ with the value given by $f(z)$. One can require $\eps_1=\eps_1(t) = \|f-g\|_\infty$ to be at most $F/2$ for sufficiently large $t$. We have:
\begin{align*}
(F - 2\eps_1)^2
&\leq \iint_{Y^2} \|g(u)-g(v)\|^2 \dd u\dd v\\
&= \sum_{y,z\in Z} \iint_{U_y\times U_z} \|g(u)-g(v)\|^2 \dd u\dd v \\
&\lesssim_{K^2} \sum_{y,z\in Z} \frac{1}{\myhash Z^2}  \|g(y)-g(z)\|^2
\end{align*}
(where $a\lesssim_c b$ means that $a\leq c b$), which, by expansion, is bounded (up to the multiplicative constant $\eta$ from \cref{super defi}) by the following:
\begin{multline*} 
 \frac{1}{\myhash Z} \sum_{y\sim z} \|g(y)-g(z)\|^2
 \leq \frac{1}{\myhash Z} \sum_{z} \left(\sum_{y\sim_1 z} \|f(y)-f(z)\|^2 
 + \sum_{y\sim_2 z}\|f(y)-f(z)\|^2 \right).
\end{multline*}
Note that by the uniform continuity of $f$ we know that the value of
\[\sup_{y\sim_1 z}\|f(y)-f(z)\| \leq \sup_{d(y,y')\leq 3/t} \|f(y)-f(y')\| \eqdef \eps_2(t)\]
goes to zero when $t$ goes to infinity, and, as the value of the first sum is bounded by $D\eps_2^2$, we can focus on the second sum.

We obtain:
\begin{align}
\frac{(F-2\eps_1)^2\eta}{K^2}-D\eps_2^2
& \leq \frac{1}{\myhash Z} \sum_z \sum_{s\in S} \sum_{y\sim_s z} \|f(y)-f(z)\|^2 \label{a} \\
&\lesssim_K \sum_{s\in S} \sum_z \sum_{y\sim_s z} \int_{U_z} \|f(y)-f(z)\|^2 \dd u.\label{b}
\end{align}
Fix $s\in S$ for a moment. Consider the open $1/t$-neighbourhood $B_s$ of the boundary $\partial \im s \subseteq Y$. Since $Y$ is a geodesic space, the complement of $B_s$ splits into the `$1/t$-interior' $I_s$ of $\im s$ and the `$1/t$-complement' $C_s$ of $\im s$:
\begin{align*}
I_s &= \{ z\in Y : B(z,1/t)\subseteq \im s \},\\
C_s &= \{ z\in Y : B(z,1/t)\subseteq Y \setminus \im s \}.
\end{align*}
For $z\in C_s$ there is no edge $y \sim_s z$, and for $z\in B_s \cap Z$ the set $U_z$ is contained in the open $2/t$-neighbourhood of $\partial \im s = \partial \dom s^{-1}$, which has measure $\eps_{3,s}(t)$ converging to zero by the assumption that $\mu(\partial \dom s^{-1})=0$.

Hence we can focus on vertices $z$ from $I_s$. In this case $U_z\subseteq \im s$ and $s^{-1}$ is defined on the whole of $U_z$, so --- for $\eps_3 \defeq D \|2f\|_\infty^2 \sum_s \eps_{3,s}$ --- we have:
\begin{align*}
\frac{(F-2\eps_1)^2\eta}{K^3}&-\frac{D\eps_2^2}{K} - \eps_3 \\
&\leq \sum_{s\in S} \sum_{z\in I_s} \sum_{y\sim_s z} \int_{U_z} \|f(y)-f(z)\|^2 \dd u \\
&\leq \sum_{s\in S} \sum_{z\in I_s} \sum_{y\sim_s z} \int_{U_z} \left(\eps_4 + \|f(s^{-1}u)-f(u)\| + \eps_1\right)^2 \dd u\displaybreak[0] \\
&\lesssim_{2D} \sum_{s\in S} \sum_{z\in I_s} \int_{U_z} \|f(s^{-1}u)-f(u)\|^2 + (\eps_1 + \eps_4)^2 \dd u,
\end{align*}
where $\eps_4(t)$ is defined as 
\[\eps_4(t) \defeq \sup_{d(y,y')\leq \frac{2L+1}{t}}\|f(y)-f(y')\| \]
and goes to $0$ as $t\to \infty$.

We have just obtained:
\begin{multline*}
\frac{(F-2\eps_1)^2\eta}{2K^3D}-\frac{\eps_2^2}{2K} - \frac{\eps_3}{2D} - \myhash S(\eps_1+\eps_4)^2  
\leq \sum_{s\in S} \int_{sY} \|f(s^{-1}u)-f(u)\|^2 \dd u,
\end{multline*}
which, after passing to infinity with $t$, yields the spectral gap:
\[\frac{\eta}{2K^3D}  \iint_{Y^2} \|f(u)-f(v)\|^2 \dd u\dd v \leq \sum_{s\in S}\int_{sY} \|f(s^{-1}u)-f(u)\|^2 \dd u.\qedhere\]
\end{proof}

\begin{proof}[Spectral gap $\implies$ expansion] Assume that $S$ has a spectral gap in $L^2(Y,\mu;X)$. In \cref{Rips-QI} we proved that, under the assumptions of \cref{main plus}, the warped cone $\cO_S Y$ is quasi-isometric to a family of graphs $(G(t))$ with uniformly bounded vertex degrees (and also unbounded cardinalities because of unboundedness of diameters). Hence, in order to complete the proof of \cref{main plus}, it suffices to prove that these graphs satisfy the Poincar\'e inequality~\eqref{expansion inequality} from \cref{super defi}.

Fix $t>0$ and the respective $Z$, $\{U_z\}_{z\in Z}$, and $G$ from \cref{construction}. For any $f\in L^2(G;X)$ we can bound from below:
\begin{align}
\myhash S \sum_{y\sim z}\|f(y) - f(z)\|^2 
&\geq \sum_{s\in S} \sum_z \sum_{y\sim_s z}\|f(y) - f(z)\|^2 \label{c} \\
&\geq \sum_{s\in S} \sum_z \sum_{y\sim_s z}\frac{\mu(s U_y\cap U_z)}{\mu(U_z)} \|f(y) - f(z)\|^2.\notag
\end{align}
If one extends $f$ to a function constant on every $U_z$, the last expression equals:
\begin{align}
\sum_{s\in S} \sum_z \frac{1}{\mu(U_z)} \sum_{y\sim_s z} &\int_{s U_y\cap U_z} \|f(s^{-1}u) - f(u)\|^2\dd u \label{d}\\
&\geq \sum_{s\in S} \sum_z \frac{\myhash Z}{K} \sum_{y\sim_s z} \int_{s U_y\cap U_z}\|f(s^{-1}u) - f(u)\|^2\dd u \notag \\
&= \frac{\myhash Z}{K} 
           \sum_{s\in S} \int_{sY\cap Y} \|f(s^{-1}u) - f(u)\|^2\dd u,\notag
\end{align}
and this, by the spectral gap property, can be bounded below as follows:
\begin{align*}
\frac{\myhash Z}{K} 
           \sum_{s\in S} \int_{sY} \|f(s^{-1}u) - f(u)\|^2\dd u
&\geq \frac{\kappa \myhash Z}{K} \iint_{Y^2} \| f(u) - f(v)\|^2 \dd u\dd v\\
&= \frac{\kappa \myhash Z}{K} \sum_{y,z\in Z} \iint_{U_y\times U_z} \| f(u) - f(v)\|^2 \dd u\dd v\\
&\geq \frac{\kappa \myhash Z}{K} \sum_{y,z\in Z} \left(\frac{1}{K \myhash Z}\right)^2 \|f(y)-f(z)\|^2.
\end{align*}
So for $\eta = \frac{\kappa}{K^3 {\scriptsize \myhash} S}$ we have just proved that
\[ \sum_{y\sim z}\|f(y) - f(z)\|^2  \geq \frac{\eta}{\myhash Z} \sum_{y,z\in Z} \|f(y)-f(z)\|^2.\qedhere\]
\end{proof}

We will now deduce \cref{cor}.

\begin{proof}[Proof of \cref{cor}] Let $(Y,d,\mu)$ be an Ahlfors regular metric measure space with an ergodic action by Lipschitz homeomorphisms of a finitely generated group~$\Gamma$ with V. Lafforgue's Banach property (T) as in \cites{Lafforgue2, Liao}. Let $(G(t))_{t>0}$ be a family of graphs with uniformly bounded vertex degrees and quasi-isometric to $\cO_\Gamma Y$, for instance the family from \cref{construction}. For any Banach space $X$ of non-trivial type, also the Bochner space $L^2(Y,\mu; X)$ has non-trivial type, equal to the type of~$X$ (see e.g.\ \cite{Diestel-Jarchow-Tonge}*{Theorem 11.12}). Consequently, the action has a spectral gap in $L^2(Y,\mu;X)$. Now, by \cref{main} (formally: together with \cref{expansion-invariant}) the family $(G(t))_{t>0}$ is an expander with respect to~$X$.
\end{proof}

\subsection{Remarks}\label{subsec:remarks} Let us finish with some remarks on the proof and a corollary.

\vspace{1mm}\noindent1. In the proof that the spectral gap for an action implies the expansion of the graphs $G(t)$, the Lipschitz condition was only used in \cref{Rips-QI} to obtain a bound on the degrees of the graphs $G(t)$, but not in the proof of the Poincar\'e inequality~\eqref{expansion inequality}. Hence, even without this assumption one still gets a `weak expander', that is, an expander with potentially unbounded degree of vertices, cf.\ \cite{expgirth1}.

\vspace{1mm}\noindent2. However, the boundedness of the degree of vertices in the definition of an expander was used crucially for the opposite implication, namely that expansion implies spectral gap.

Indeed, even for $z\in I_s$ one cannot bound $\frac{1}{{\scriptsize \myhash} Z}\|f(y)-f(z)\|^2$ in line \eqref{a} from above by the integral $\int_{sU_y\cap U_z} \|f(y)-f(z)\|^2\dd u$ (like we did for the opposite bound in lines \eqref{c} to \eqref{d}) because $\mu(sU_y\cap U_z)$ can be arbitrarily small. Consequently, in line \eqref{b} for every $y$ (such that $y\sim_s z$) we integrate the same function over the whole of $U_z$, so we later need boundedness of $D$ in order to control the number of these integrals.

Formally, we also used the Lipschitz continuity of $s\in S$ to show the vanishing of $\eps_4$, but uniform continuity would suffice.

\vspace{1mm}\noindent3. In the proof that the expansion of $G(t)$ implies a spectral gap for the action, it suffices to assume that $(G(t))_{t\in T}$ is an expander for some unbounded subset $T\subseteq (0,\infty)$.

\vspace{1mm}\noindent4. We also did not rely essentially on the exponent $p=2$, that is, the equivalence of the $L^p$-analogues of \eqref{expansion inequality} and \eqref{inequality ultimate plus} holds under the assumptions of \cref{main plus} for any exponent $p\in[1,\infty)$. (For $p=\infty$ the analogues of \eqref{expansion inequality} and \eqref{inequality ultimate plus} never hold unless $X$ is a point.)

Being an expander with respect to $X$ does not depend on $p$ by the result of Cheng \cite{Cheng}, generalising partial results of Mimura \cite{Mimura}, so we reach the following conclusion.
\begin{cor}\label{cor in rems} Under the hypotheses of \cref{main plus}, the spectral gap condition in $L^p(Y,\mu;X)$ does not depend on $p\in [1,\infty)$.
\end{cor}

\vspace{1mm}\noindent5. In \cref{warped} we require $Y$ to be geodesic in order to guarantee that the family $(Y,td)_{t>0}$ and hence also the family $(tY,d_S)_{t>0}$ are quasi-geodesic. Since \cref{main plus} compares the family $(tY,d_S)_{t>0}$ to graphs and being quasi-isometric to a graph is equivalent to being quasi-geodesic, the assumption is rather unavoidable. See \cite{Sawicki-cBCc} for warped cones (over spectral gap actions on non-geodesic spaces) being coarsely non-equivalent to any family of graphs.

\vspace{1mm}\noindent6. Nonetheless, the equivalence of the spectral gap for the action and the expansion for graphs $G(t)$ still holds, irrespective of whether $Y$ is geodesic (the same remark applies to \cref{cor in rems}). It is also true for graphs $G'(t)$ obtained by allowing only edges of type $(2)$ of graphs $G(t)$. Similar graphs were introduced in \cite{Vigolo} under the name of approximating graphs for the action, and the respective equivalence was obtained in the classical setting.

In particular, even without the geodesic assumption, expansion with respect to Banach spaces of non-trivial type holds for graphs $G(t)$ and $G'(t)$ and actions as in \cref{cor}.

\section*{Acknowledgements} The author is grateful to Alan Czuro\'n and Masato Mimura for inspiring conversations on Banach-space expanders and to Pawe\l{} J\'oziak and Mateusz Wasilewski for discussions on Bochner spaces. The author wishes to express his gratitude to Piotr Nowak for his constant support and guidance. The author is indebted to Pawe\l{} J\'oziak, Piotr Nowak, and Federico Vigolo for remarks on the manuscript and to the anonymous referee for their careful reading and helpful comments. The author thanks Noémie Combe for help in preparing French versions of the abstract and keyword list.

The author was partially supported by the Narodowe Centrum Nauki grant Preludium number 2015/19/N/ST1/03606 and by the Max-Planck-Gesellschaft.

\vspace{.5cm}
\end{document}